\documentclass[12pt]{amsart}
\usepackage{float}
\usepackage{amsfonts,amssymb,amsmath,textcomp}
 2

\usepackage[autostyle]{csquotes}
\theoremstyle{plain}
\newtheorem{thm}{Theorem}[section]

\newtheorem{lem}[thm]{Lemma}
\newtheorem{prop}[thm]{Proposition}

\newtheorem{con}[thm]{Conjecture}

\theoremstyle{definition}

\newcommand{\N}{\mathbb{N}}
\newcommand{\Z}{\mathbb{Z}}
\newcommand{\Q}{\mathbb{Q}}

\begin{document}
\baselineskip 6.1mm

\title
{Square-free values of polynomials}
\author{Prem Prakash Pandey}

\address[Prem Prakash Pandey]{Indian Institute of Science Education and Research, Berhampur, India.}
\email{premp@iiserbpr.ac.in}


\date{\today}


\begin{abstract}

It is conjectured that all separable polynomials with integers coefficients, satisfying some local conditions, take infinitely many square free values on integer arguments. But not a single polynomial of degree greater than $3$ is proven to exhibit this property. In this article, we propose a method to show that ``cyclotomic polynomials $\Phi_{\ell}(X)$ take square free values with positive proportion".  Following this method, conditionally, we do prove the Square-free conjecture for all cyclotomic polynomials. The method is applicable to some other class of polynomials as well. Moreover we show that the method readily succeeds for quadratic polynomials to give unconditional result. In the appendix, we give an elementary proof of the square-free conjecture for cyclotomic polynomials under abc conjecture.\\

\end{abstract}

\maketitle

\section{Introduction}
Let $f(X)\in \mathbb{Z}[X]$ be a separable polynomial (i.e. having no repeated roots) and let $m \geq 2$ be an integer. It is conjectured that $f(X)$ takes infinitely many $m-$ power free values if there is no local obstruction; namely, for each prime $p$ there is some integer $a$ such that $f(a)$ is not divisible by $p^m$. There is a lot of literature on this conjecture and some variations (stronger conjectures) of it \cite{MB,PE,TE,AG,HH,HH1,CH1,CH2,MN,TR}. For $m=2$ we mention the conjecture more precisely \cite{Zeev}.\\

\begin{con}\label{Con1}
 Let $f(X) \in \mathbb{Z}[X]$ be a separable polynomial of degree $g\geq 1$. Assume that $gcd\{f(n): n \in \mathbb{Z}\}$ is square-free. Then there are infinitely many square free values taken by $f(n)$, in fact a positive proportion of the values are square-free:
$$|\{1 \leq n \leq X: f(n) \mbox{ is square free }\}| \sim C_fX, \mbox{ as } X\longrightarrow \infty,$$
with 
\begin{equation}\label{C1}
C_f=\prod_p \left(1-\frac{\delta_f(p^2)}{p^2}\right),
\end{equation}
where $\delta_f(N)=|\{a \pmod N: f(a)=0 \pmod N\}|.$
\end{con}

It is worthwhile to remark that Granville has established the Conjecture \ref{Con1} under the abc conjecture \cite{AG}. Without any assumption, conjecture \ref{Con1} is known when $g\leq 3$ due to efforts of various Mathematicians\cite{PE,TE, CH1, CH2, GR}. On the other hand, not a single polynomial $f(X) \in \mathbb{Z}[X]$ of degree $g\geq 4$ is known for which which infinitude of square free values is established, see page 1 in \cite{MF} or \cite{MB}.\\

In this article we propose a roadmap to establish the Conjecture \ref{Con1} for some family of polynomials (in particular cyclotomic polynomials and the widely talked polynomial $X^4+2$). Furthermore, we demonstrate that the method proposed in the article is fully realized for quadratic polynomials, thus providing  a proof of Conjecture \ref{Con1} for quadratic polynomials (see section 5).  In our approach it follows as a straight forward application of inclusion-exclusion principle. \\

Let $\ell$ be an odd prime and $\zeta_{\ell}$ denote a fixed primitive $\ell^{th}$ root of unity. For us, $\Phi_{\ell}(X)$ denotes $\ell^{th}$ cyclotomic polynomial. For any natural number $n$ we define
$$R_{\ell}(n):=min \{d \in \N: n| \Phi_{\ell}(d)\},$$
if there are integers $d \in \N$ such that $n| \Phi_{\ell}(d)$ else we put $R_{\ell}(n)=\infty.$ It is immediate to see that $R_{\ell}(n) \leq R_{\ell}(n^2)$, whenever $R_{\ell}(n)$ is finite. Throughout the article, we use symbol $\mathbb{P}$ to denote the set of primes $p$ which appear as a divisor of $\Phi_{\ell}(d)$ for some integer $d \in \N$. For positive integers $i$ and a real parameter $T$ we put
$$S^i(\ell, T):=\{p \in \mathbb{P}: R_{\ell}(p^i) \leq T\}.$$
The sizes of $S^i(\ell, T)$ are of arithmetical significance (in this article we use sieving along elements of $S^i(\ell, T)$ to prove some arithmetical results). On the size of $S^1(\ell, T)$, we prove the following theorem.
\begin{thm}\label{T31}
 The following upper bound on the size of $S^1(\ell, T)$ holds $$|S^1(\ell, T)|<<T.$$
 \end{thm}
The size of $S^2(\ell, T)$ is very important for our purpose. Based on some empirical studies we propose the following estimate.
\begin{con}\label{Con2}
 The following holds: 
$$|S^2(\ell,T)| =o(T).$$
\end{con}
In Section 5 we present some heuristics in support of Conjecture \ref{Con2}. We present two different supporting arguments for Conjecture \ref{Con2}. The following inequality seems to hold true for most of the primes $p$
\begin{equation}\label{SC21}
R_{\ell}(p^2)>p.
\end{equation}
Under the validity of equation (\ref{SC21}), it it easy to see that the Conjecture \ref{Con2} is true. Further elaborations are presented in Section 5.\\

Now we mention the main result of the article.
\begin{thm}\label{ST1}
If the Conjecture \ref{Con2} is true then the Conjecture \ref{Con1} holds for cyclotomic polynomials $\Phi_{\ell}(X)$.
\end{thm}

In Section 2 we prepare the background for the proof of Theorem |ref{ST1}.  In Section 3 we give the proof of Theorem \ref{ST1}. In Section 4 we demonstrate that the techniques we use to prove Theorem \ref{ST1} apply to some other polynomials too. This is illustrated by taking the example of the, widely popular, polynomial $X^4+2$. The section 5 lists some arguments in favour of Conjecture \ref{Con2}. In particular, the Conjecture \ref{Con2} is shown to be true for quadratic polynomials. The article contains an appendix where we include our proof of Conjecture \ref{Con1} for cyclotomic polynomials under abc conjecture.Our proof is conceptually does not involve all the machinery used by Granville in \cite{AG}.\\

\noindent{\bf Acknowledgment.}
I am very much thankful to Prof. Manjul Bhargava for going through an earlier version and making some very good suggestions, and pointing out some technical flaws. Thanks are also due to Prof. Preda Mihailescu for some important comments and encouragement.
I am indebted to my friend Dr. Umesh Dubey for helping me to improve the presentation. 

\vspace{.5cm}
\section{Preparation for Theorem \ref{ST1}}
The $norm$ for the extension $\Q(\zeta_{\ell})/ \Q$ will be abbreviated by $N$. We begin with the fact 
\begin{equation}\label{SE0}
\Phi_{\ell}(X)=\prod_{i=1}^{\ell-1} (X-\zeta_{\ell}^i)=N(X-\zeta_{\ell}).
\end{equation}
Let $d \in \Z$ be such that $\Phi_{\ell}(d)$ is divisible by $p^2$ for a prime $p$. Then a prime $\mathfrak{p}$ of $\Z[\zeta_{\ell}]$, dividing $p$, divides $d-\zeta_{\ell}$ and at least one of the following holds:\\
(a) the residue degree of $\mathfrak{p}$ is not one,\\
(b) $\mathfrak{p}^2|d-\zeta_{\ell}$,\\
(c) there is a non-trivial conjugate $\sigma(\mathfrak{p})$ of $\mathfrak{p}$ such that both  $\sigma(\mathfrak{p})$ and $\mathfrak{p}$ divide $d-\zeta_{\ell}$.\\

We will show that the possibilities (a) and (c) do not occur. Thus, the problem `$\phi_{\ell}(X)$ takes infinitely many square free values' reduces to the `linear polynomial $X-\zeta_{\ell}$ takes infinitely many square free values as $X$ runs in $\Z$'. 
\begin{prop} \label{SP1}
Let $d$ be an integer. Then either $d- \zeta_{\ell}$ is a unit or all the prime divisors $p$ of $\Phi_{\ell}(d)=N(d- \zeta_{\ell})$ have residue degree one in $\Q(\zeta_{\ell})$. 
\end{prop}
\begin{proof}
If $d- \zeta_{\ell}$ is a unit then nothing to prove. Assume that $d-\zeta_{\ell}$ is not a unit and let $p$ be a prime divisor of $N(d- \zeta_{\ell})$. Since $\ell$ has residue degree one in $\Q(\zeta_{\ell})$, the proposition holds if $p = \ell$. Assume that $p \neq \ell$ and fix a prime $\mathfrak{p}$ above $p$ such that $\mathfrak{p} | d- \zeta_{\ell}^i$ for some $i$.\\ 
If $p \not\equiv 1 \pmod {\ell}$ then there is a non trivial element $\sigma$ in the $Gal(\mathbb{Q}(\zeta_{\ell})/ \mathbb{Q})$ which fixes $\mathfrak{p}$. If $\sigma(\zeta_{\ell}^i)=\zeta_{\ell}^j$ then we have\\
$$\mathfrak{p} |d-\zeta_{\ell}^i \mbox{ and } \mathfrak{p} |d-\zeta_{\ell}^j,$$ for $i \neq j$. Consequently we get $\mathfrak{p} |(\zeta_{\ell}^i-\zeta_{\ell}^j)$ and, hence, $\mathfrak{p}=<1-\zeta_{\ell}>$. This forces $p=\ell$, which is a contradiction. Hence $p \equiv 1 \pmod {\ell}$, and $p$ has residue degree one.
\end{proof}
Note that there can be at most $2(\ell-1)$ values of $d$ for which $d-\zeta_{\ell}$ is a unit. As $d-\zeta_{\ell}$ being a unit will give $N(d-\zeta_{\ell})= \pm1$, and left side is a polynomial of degree $\ell-1$ in $d$. Thus, if $p \neq \ell$ is a prime divisor of $\phi_{\ell}(d)$ for some integer $d$ then $p$ is of residue degree one. Using the factorization of $p$ in $\Q(\zeta_{\ell})$, we see that if $\mathfrak{p}$ is a prime ideal above $p$ then $\mathfrak{p}=<p,d-\zeta_{\ell}>$ for some integer $d \in \{1, \ldots, p\}$. From this it follows that $R_{\ell}(p)=d \leq p$ unless $d=d_p$. In the latter case we see that $R_{\ell}(p) \leq 2p$.\\
\begin{prop}\label{SP11}
Let $d$ be an integer. If $\mathfrak{p}$ is a prime ideal dividing $d-\zeta_{\ell}^i$ then no non-trivial conjugate of $\mathfrak{p}$ divides $d-\zeta_{\ell}^i$.
\end{prop}
\begin{proof}
In case $\mathfrak{p}=<1-\zeta_{\ell}>$, then nothing to prove. So we assume that $\mathfrak{p} \neq <1-\zeta_{\ell}>$. From the Proposition \ref{SP1} we see that $\mathfrak{p}$ has $\ell-1$ conjugates. If there is a non-trivial $\sigma \in Gal(\Q(\zeta_{\ell})/ \Q)$ such that $\mathfrak{p}$ and $\sigma(\mathfrak{p})$ contain $d-\zeta_{\ell}^i$, then $\mathfrak{p}$ contains $d-\zeta_{\ell}^i$ and $\sigma^{-1}(d-\zeta_{\ell}^i)=d-\zeta_{\ell}^j$ for some $j \neq i$. Again, as in Proposition \ref{SP1} , we see that $\mathfrak{p}=<1-\zeta_{\ell}>$. But we have assumed $\mathfrak{p} \neq <1-\zeta_{\ell}>$. This proves the Proposition.
\end{proof}
Thus, to show that $\Phi_{\ell}(X)$ takes positive proportion of square free values it is necessary and sufficient to show that for a positive proportion of integers $d$ the elements $d-\zeta_{\ell}$ are square free. \\

At this instant, it is appropriate to mention the generalization of the Conjecture 1 to number fields (see the work of Hinz \cite{JH1}). It is shown by Hinz that if $g(X)$ is a polynomial of degree 3 or less with coefficients in the ring of integers $\mathbb{O}_K$ of a number field $K$ then $g(X)$ takes a positive proportion of square-free values as the arguments runs in $\mathbb{O}_K$. If one can show that $g(X)$ takes positive proportion of square-free values as the arguments run in $\Z$ then that will provide another way to prove Conjecture \ref{Con1}. \\

In this and the next section $p$ will stand for a generic element in $\mathbb{P}$ and $\mathfrak{p}$ will denote a prime divisor of $p$ in $\Z[\zeta_{\ell}]$. From the Proposition \ref{SP1} it follows that $\mathfrak{p}$ has residue degree one. We have the following elementary Lemma.
\begin{lem}\label{L1}
We have 
$$\frac{\mathbb{Z}[\zeta_{\ell}]}{\mathfrak{p}^2} \cong \frac{\mathbb{Z}}{ {p^2} \mathbb{Z}}.$$
\end{lem}
\begin{proof}
We compose the quotient map $$\mathbb{Z}[\zeta_{\ell}] \longrightarrow \mathbb{Z}[\zeta_{\ell}]/ \mathfrak{p}^2$$ with the inclusion map $$\mathbb{Z}\hookrightarrow \mathbb{Z}[\zeta_{\ell}].$$ 
This results in a homomorphism $$\phi: \mathbb{Z} \longrightarrow \mathbb{Z}[\zeta_{\ell}]/ \mathfrak{p}^2.$$ 
If $n$ is in the kernel of $\phi$, then $\mathfrak{p}^2 |n$ and hence $p^2|n$. Thus kernel of $\phi$ is contained in $p^2 \Z$ and certainly $p^2 \Z $ is in the kernel of $\phi$. Thus the  kernel of $\phi$ is $p^2 \mathbb{Z}$. Since the residue degree of $\mathfrak{p}$ is one, we have $|\mathbb{Z}[\zeta_{\ell}]/ \mathfrak{p}^2|=p^2$. Hence $\phi$ induces the desired isomorphism.
\end{proof}
We have another elementary Lemma
\begin{lem}\label{L2}
Let $T$ be a positive real number. The cardinality of the set $\{ d \in \mathbb{N} :  d<T \mbox{ and }\mathfrak{p}^2  \mbox{ divides }d-\zeta_{\ell} \}$ is less than or equal to $T/p^2+1$.
\end{lem}
\begin{proof}
Let $d \in \mathbb{N}$ be such that $ \mathfrak{p}^2 |d-\zeta_{\ell} $. Then for any integer $t$ with $|t|<p^2$, $\mathfrak{p}^2 \nmid d+t-\zeta_{\ell}$. The lemma follows at once.  
\end{proof}

%

Now we establish an upper bound on the order of $S^1(\ell,T)$, as claimed in the introduction.
\begin{prop}\label{P31}
For any fixed real number $M$ there are at most $2 \ell M$ many primes $p \in \mathbb{P}$ with $R_{\ell}(p) \leq M$.
\end{prop}
\begin{proof}
It is sufficient to show that for any positive integer $u<M$ the number of primes $p$ with $R_{\ell}(p)=u$ is bounded above by a fixed constant. We fix a positive integer $u<M$.\\
Note that $R_{\ell}(p)=u$ implies that $p| \Phi_{\ell}(u)$ and $u<2p$. There can be at most $2\ell$ primes $p$ satisfying $u<2p$ and $p| \phi_{\ell}(u)$.  Consequently there can be at most $2 \ell$ many primes $p$ with $R_{\ell}(p)=u$. Thus there can be at most $2 \ell M$ primes $p$ with $R_{\ell}(p) \leq M$.
\end{proof}
\begin{proof}
(Theorem \ref{T31})
From the definition, it is evident that a prime $p$ is in $S^1(\ell, T)$ if and only if $R_{\ell}(p) \leq T$. Now the theorem follows from the Proposition \ref{P31}.

\end{proof}

\vspace{.5cm}

\section{Proof of Theorem \ref{ST1}}
In this section we prove Theorem \ref{ST1}. The proof uses sieve method. It is here the invariants $R_{\ell}(n)$ demonstrate their strength. For $f(X)=\Phi_{\ell}(X)$ we write $C_{\ell}:=C_f$. 

\begin{proof}[Proof of Theorem \ref{ST1}]
Let $S$ be the set of all positive integers $d \in \N$ such that $d-\zeta_{\ell}$ is square free. The density mentioned in Theorem \ref{ST1} is the density of the set $S$; from Proposition \ref{SP1} and Proposition \ref{SP11} it follows that, for any $d \in S$ the value $\Phi_{\ell}(d)$ is square free and converse is also true. We let $S_{\mathfrak{p}}$ denote the set of all positive integers $d \in \N$ such that $\mathfrak{p}^2 \nmid (d-\zeta_{\ell})$. Then it follows that 
\begin{equation}
S=\bigcap_{\mathfrak{p}} S_{\mathfrak{p}}.
\end{equation}
Now we proceed to prove that $C_{\ell}$ is an upper bound on the density of integers $d$ such that $\Phi_{\ell}(d)$ is square-free. For $p \in \mathbb{P}$ there are precisely $\ell$ integers, say, $a_1, \ldots, a_{\ell}$ which are non congruent to each other modulo $p^2$ and are solution of 
$$X^{\ell}=1 \pmod {p^2}.$$ 
If for a positive integer $d$ we have $\mathfrak{p}^2|d-\zeta_{\ell}$, then by Lemma \ref{L1} 
$$d \equiv a_j \pmod {p^2} \mbox{ for some } j.$$
Thus the density of $S_{\mathfrak{p}}$ is 
$$(p^2-\ell)/p^2=(1-\ell/p^2).$$ 
Similarly, for $p = \ell$ the density of $S_{\mathfrak{p}}$ is $(1-1/p^2)$.\\
On the density of the set $S$ we make the following claim.\\
{\bf Claim: } The density of the set $S$ is  
$$\prod_p  (1-A_p/p^2),$$
where $A_p=\ell$ for $p \equiv 1 \pmod {\ell}, A_{\ell}=1$ and $A_p=0$ otherwise. Since $A_p<p^2$, the infinite product $$ \prod_p (1-A_p/p^2)$$ converges to a non-zero value (see the remark at the end of section 2 in \cite{HH1}).

Now let $T$ be a (arbitrary large) real parameter and $S(T)=\{d \in S: d<T\}$. First we shall show that 
\begin{equation}\label{SE1}
\limsup_{T \rightarrow \infty} \frac{|S(T)|}{|\{d<T\}|} \leq \prod_p  (1-A_p/p^2).
\end{equation}
In what follows, $\mathbb{T} $ will be a finite set of primes of $\Z[\zeta_{\ell}]$ which is closed under conjugation (that is, for $\mathfrak{p} \in \mathbb{T}$ all the conjugates of $\mathfrak{p}$ are also in $\mathbb{T}$. It is important to assume that $\mathbb{T}$ is closed under conjugation as for $p=11$ both $3-\zeta_5$ and $9-\zeta_5$ have square factors and $|3-9|<11$). Then, by Chinese remainder theorem, it follows that the density of non-negative integers $d $ such that $d-\zeta_{\ell}$ is not divisible by $\mathfrak{p}^2$ for any prime ideal $\mathfrak{p} \in \mathbb{T}$ is
$$\prod_{\substack{ p\\\ \mathfrak{p} \in \mathbb{T}}}(1-A_p/p^2).$$
Thus for any $+ve$ real number $M$ one has
\begin{equation}\label{SE101}
 \limsup_{T \rightarrow \infty} \frac{|\{d<T: d \in \bigcap_{\substack {\mathfrak{p} \in \mathbb{T} \\\ p<M}} S_{\mathfrak{p}} \}|}{|\{d<T\}|}  = \prod_{ p<M}(1-A_p/p^2).
 \end{equation}
Let $\epsilon>0$ be arbitrary. Since the infinite product $\prod_{p}(1-A_p/p^2)$ converges, there is a real number $M$ such that
$$\prod_{p}(1-A_p/p^2) >\prod_{ p<M}(1-A_p/p^2)-\epsilon.$$
As $\epsilon>0$ is arbitrary, from equation (\ref{SE101}) it follows that 
$$\limsup_{T \rightarrow \infty} \frac{|S(T)|}{|\{d<T\}|} \leq \prod_p  (1-A_p/p^2),$$
proving the inequality (\ref{SE1}). Thus, it will follow that the density of $S$ is $\prod_{ p}(1-A_p/p^2),$ provided we can establish
\begin{equation}\label{SE11}
\liminf_{T \rightarrow \infty} \frac{|S(T)|}{|\{d<T\}|} \geq \prod_{p} (1-A_p/p^2).
\end{equation}
%
Observe that 
$$\bigcap_{R_{\ell}(p) <M} S_{\mathfrak{p}} \subset S \bigcup \left( \bigcup_{R_{\ell}(p) \geq M} \bar{S_{\mathfrak{p}}}\right),$$ where $\bar{S_{\mathfrak{p}}}=\{d \in \mathbb{N}: d-\zeta_{\ell} \in \mathfrak{p}^2\}$ and $M>1$ is a positive real number. Hence for any positive real number $T>1$ we have
$$|\{d<T: d\in \bigcap_{R_{\ell}(p) <M} S_{\mathfrak{p}}\}| \leq |\{d<T:d \in S \}|+\sum_{R_{\ell}(p) \geq M}|\{d<T: d \in \bar{S_{\mathfrak{p}}}\}|.$$ 
Thus we have
\begin{equation}\label{SE30}
|\{d<T:d \in S \}| \geq |\{d<T: d\in \bigcap_{R_{\ell}(p) <M} S_{\mathfrak{p}}\}|-\sum_{R_{\ell}(p) \geq M}|\{d<T: d \in \bar{S_{\mathfrak{p}}}\}|
\end{equation}
Note that if $d<T$ and $d \in \bar{S_{\mathfrak{p}}}$ then $\mathfrak{p}^2|d-\zeta_{\ell}$ for some $d<T$ and consequently then $R_{\ell}(p^2) < T$. Further, since $R_{\ell}(p) \leq R_{\ell}(p^2)$, from inequality (\ref{SE30}) we obtain the following inequality.
\begin{equation*}
|\{d<T:d \in S \}| \geq |\{d<T: d\in \bigcap_{R_{\ell}(p) <M} S_{\mathfrak{p}}\}|-\sum_{T>R_{\ell}(p^2) \geq M}|\{d<T: d \in \bar{S_{\mathfrak{p}}}\}|.
\end{equation*}
Consequently, we obtain
\begin{equation*}
\frac{|\{d<T:d \in S \}|}{|\{d<T\}|} \geq \frac{|\{d<T: d\in \bigcap_{R_{\ell}(p) <M} S_{\mathfrak{p}}\}|}{|\{d<T\}|}-\sum_{T>R_{\ell}(p^2) \geq M} \frac{|\{d<T: d \in \bar{S_{\mathfrak{p}}}\}|}{|\{d<T\}|}.
\end{equation*}
Now using Lemma \ref{L2}, we get 
\begin{equation*}
\frac{|\{d<T:d \in S \}|}{|\{d<T\}|} \geq \frac{|\{d<T: d\in \bigcap_{R_{\ell}(p) <M} S_{\mathfrak{p}}\}|}{|\{d<T\}|}-\sum_{T > R_{\ell}(p^2) \geq M} \left( \frac{1}{p^2}+\frac{1}{T} \right).
\end{equation*}
(Ideally, we shall put $(\ell-1)$ in front of the last sum as for any prime $p$, in the range, there are $(\ell-1)$ many sets $\bar{S_{\mathfrak{p}}}$ involved. However the above inequality remains true without this factor) This gives us
\begin{equation}\label{SE31}
\liminf_{T\rightarrow \infty} \frac{|\{d<T: d \in S\}|}{|\{d<T\}|} \geq \prod_{p <M}(1-A_p/p^2)-\sum_{T>R_{\ell}(p^2) \geq M}\left(\frac{1}{p^2}+\frac{1}{T}\right).
\end{equation}
If Conjecture 2 holds then the number of terms in the sum on the right side of inequality (\ref{SE31}) is of order $o(T)$. Since $\sum_{R_{\ell}(p^2)\geq M}\frac{1}{p^2}$ is tail of a convergent series and $\prod_{p}(1-A_p/p^2)$ converges, by taking $M$ and $T$ arbitrarily large in  inequality (\ref{SE31}) it follows that
\begin{equation*}
\liminf_{T\rightarrow \infty} \frac{|\{d<T: d \in S\}|}{|\{d<T\}|} \geq \prod_{p}(1-A_p/p^2),
\end{equation*}
proving the inequality (\ref{SE11}). This finishes the proof of Theorem \ref{ST1}.
\end{proof}

%

\vspace{1cm}

\vspace{.5cm}

\section{The polynomial $X^4+2$}
In this section we apply our methods to the polynomial $X^4+2$. Treating polynomial $X^4+2$ explicitly illustrates `for what polynomials our methods can be applied'. First we prove an analogue of Proposition \ref{SP1}. Let $\beta \in \mathbb{C}$ denote a root of $f(X)=X^4+2$ and let $\beta_1=\beta, \ldots, \beta_4$ denote all the roots of $f(X)$. Put $K=\mathbb{Q}(\beta)$ and $d_f=\prod_{i<j}(\beta_i-\beta_j)^2$, the discriminant of $f(X)$. Note that the splitting field of $f(X)$ is $L:=\Q(i, \beta)$.\\

Let $p$ be a rational prime not dividing $d_f=2^9$ and $\mathbf{p}$ be a prime in $\Z[i]$ above $p$. Then we have the following lemma.
\begin{lem}\label{L5}
The prime $p$ splits completely in the extension $K/ \Q$ if and only if $\mathbf{p}$ splits completely in the extension $L/ \Q(i)$.
\end{lem}
 \begin{proof}
 Since $p\nmid d_f$, Dedekind theorem on factorization of primes (see Theorem 3.8.2 in \cite{HK}) is applicable.\\
If $p \equiv 1 \pmod 4$ then $\Z[i]/\mathbf{p} \cong \Z / p\Z$. Thus $X^4+2 \pmod {\mathbf{p}}$ splits completely if and only if $X^4+2 \pmod p$ splits completely. Hence, with an application of Dedekind theorem, we are through in this case.\\
If $p \equiv 3 \pmod 4$ then $\Z[i]/\mathbf{p} \cong \Z / p^2\Z$. Certainly if $p$ splits completely in $K/ \Q$ then $\mathbf{p}$ splits completely in $L/ \Q(i)$ (by checking residue degree). On the other hand, if $\mathbf{p}$ splits completely in $L/ \Q(i)$ then by Dedekind theorem $X^4+2 \pmod {\mathbf{p}}$ factors into distinct linear factors. Thus $X^4+2 \pmod {p^2}$ factors into distinct linear factors. Certainly $X^4+2 \pmod p$ factors linearly. Since $p \neq 2$, we see that $X^4+2 \mod p$ factors in distinct linear factors. Again, using Dedekind theorem, we are through in this case.
 \end{proof}
For any integer $d$, let $\mathfrak{p}$ be a prime of $\mathbb{O}_K$, lying above $p$, dividing $d-\beta$. We contend that $\mathfrak{p}$ is either a divisor of $d_f$ or is of residue degree $1$. Towards this, let $\wp$ be a prime of $\mathbb{O}_L$, lying above $\mathfrak{p}$, dividing $d-\beta$. If $\wp$ has residue degree more than $1$ in the extension $L/ \Q(i)$, then there is a non-trivial $\sigma \in Gal(L/ \Q(i))$ such that $\sigma(\wp)=\wp$. Consequently 
$$\wp |d-\beta \mbox{ and } \wp|d-\sigma(\beta).$$
Thus $\wp|d_f$. 
Since $2$ is totally ramified in $K/\Q$, we have proved the following analogue of Proposition \ref{SP1}.
\begin{prop}\label{SP2}
For any $d\in \mathbb{Z}$, all the prime ideals of $\mathbb{O}_K$ dividing $d-\beta$ have residue degree $1$.
\end{prop}
Along the similar lines, we can prove the following.
\begin{prop}\label{SP21}
For $d \in \Z$ if $\mathfrak{p}$ is a prime ideal dividing $d-\beta$ then no non-trivial conjugate of $\mathfrak{p}$ can divide $d-\beta.$
\end{prop}

Before proceeding further, we remark that an analogue of Lemma \ref{L1} is true in general, i.e.
\begin{lem}\label{L3}
For any number field $K$ and a prime $\mathfrak{p}$ of $\mathbb{O}_K$ of residue degree $1$ we have an isomorphism 
$$\frac{\mathbb{O}_K}{\mathfrak{p}^2} \cong \frac{\Z}{p^2\Z},$$
with $p\Z = \mathfrak{p} \cap \Z$.
\end{lem}
Also an analogue of Lemma \ref{L2} is true, i.e.
\begin{lem}\label{L4}
Let $T$ be a positive real number, arbitrarily large. For any prime ideal $\mathfrak{q}$ of $\mathbb{O}_K$ and $q\Z=\mathfrak{q} \cap \Z$, number of $d \in \N$ with $d<T$ and $\mathfrak{q}^2|d-\beta$ is less than or equal to $T/q^2+1$.
\end{lem}
Let $\mathbb{P}_K$ denote the set of all the primes in $\mathbb{O}_K$ which appear as a divisor of $f(d)$ for some integer $d$. Let $\mathfrak{p} \in \mathbb{P}_K$ and $p\Z=\Z \cap \mathbb{O}_K$. Then $\mathfrak{p}$ is of residue degree one. Analogous to $R_{\ell}(n)$ we define 
$$R_f(n)=\min \{|d|: n|f(d)\},$$
in case there is some integer $d$ with $f(d)$ being divisible by $n$ else we put $R_f(n)=\infty$. Further, for a real parameter $T$ and positive integers $i$ we put
$$S^i(f,T)=\{p: R_f(P^i) \leq T\}.$$

\begin{thm}\label{ST2}
Assume that the following estimate is true
$$|S^2(f,T)|=o(T).$$ Then the density of $d \in \N$ such that $f(d)$ is square free is positive and equals to the conjectured density.
\end{thm}

\begin{proof}
Let $S_K$ denote the set of all $d \in \mathbb{N}$ such that $d-\beta$ is square free. For each prime $\mathfrak{p} \in \mathbb{P}_K$, let $S_{\mathfrak{p}}$ denote the set of $d \in \N$ such that $\mathfrak{p}^2 \nmid d-\beta$. Clearly $$S_K=\bigcap_{\mathfrak{p} \in \mathbb{P}}S_{\mathfrak{p}}.$$
For any prime $\mathfrak{p} \in \mathbb{P}_K$, let $a_1, \ldots, a_{t(p)}$ be all the elements of $\Z/ p^2\Z$ such that 
$$a_i^4+2 =0 \pmod {p^2}, \mbox{ for }1 \leq i \leq t(p).$$
Note that $t(p) \leq 4$.
If $d$ is an integer such that 
$$d \not\in \{a_1, \ldots, a_{t(p)}\},$$
then, by Lemma \ref{L3}, $\mathfrak{p}^2 \nmid d-\beta$. Let $A_p=|\{a_1, \ldots, a_{t(p)}\}|$, then the density of $S_{\mathfrak{p}}$ is $$1-\frac{A_p}{p^2}.$$
Now, by Chinese remainder theorem, for any parameter $T>1$ and $M>1$ we have
\begin{equation}\label{5}
|\{d<T: d\in \bigcap_{\substack{\mathfrak{p} \\\ N(\mathfrak{p}) \leq M}} S_{\mathfrak{p}} \}|=T \prod_{\substack{p \\\ N(\mathfrak{p}) \leq M}} \left(  1-\frac{A_p}{p^2}\right)+o(T).
\end{equation}
It is easy to see that $A_p<p^2$ and thus the infinite product $$\prod_{p} \left(  1-\frac{A_p}{p^2}\right)$$ converges. Consequently we obtain
\begin{equation}\label{E41}
\limsup_{T \longrightarrow \infty} \frac{|\{d<T: d\in S_K\}|}{|\{d<T\}|} \leq \prod_{p} \left(  1-\frac{A_p}{p^2}\right)
\end{equation}
Observe that 
$$\bigcap_{R_f(p) <M} S_{\mathfrak{p}} \subset S_K \bigcup \left( \bigcup_{R_f(p) \geq M} \bar{S_{\mathfrak{p}}}\right),$$ where $\bar{S_{\mathfrak{p}}}=\{d \in \mathbb{N}:  d-\beta \in \mathfrak{p}^2\}$. Thus,
$$|\{d<T: d\in \bigcap_{R_f(p) <M} S_{\mathfrak{p}}\}| \leq |\{d<T:d \in S_K \}|+\sum_{R_f(p) \geq M}|\{d<T: d \in \bar{S_{\mathfrak{p}}}\}|.$$ 
We note that if $d<T$ and $\mathfrak{p}^2|d-\zeta_l$ then $R_f(p^2) < T$, and, thus
$$|\{d<T:d \in S_K \}| \geq |\{d<T: d\in \bigcap_{R_f(p) <M} S_{\mathfrak{p}}\}|-\sum_{T>R_f(p^2) \geq M}|\{d<T: d \in \bar{S_{\mathfrak{p}}}\}|.$$
Now using Lemma \ref{L4} we see that 
\begin{equation*}
\frac{|\{d<T:d \in S_K \}|}{|\{d<T\}|} \geq \frac{|\{d<T: d\in \bigcap_{R_f(p) <M} S_{\mathfrak{p}}\}|}{|\{d<T\}|}-\sum_{T>R_f(p^2) \geq M}\left(\frac{1}{p^2}+\frac{1}{T}\right).
\end{equation*}
Since $\sum_{R_l(p)\geq M}\frac{1}{p^2}$ is tail of a convergent series, the product $\prod_{\mathfrak{p} \in \mathbb{P}}(1-A_p/p^2)$ converges and $|S^2(f,T)|=o(T)$, by taking $T$ and $M$ arbitrarily large we see that 
\begin{equation}\label{E42}
\liminf_{T\rightarrow \infty} \frac{|\{d<T: d \in S_K\}|}{|\{d<T\}|} \geq \prod_{\substack{p\\\ \mathfrak{p} \in \mathbb{P}}}(1-A_p/p^2).
\end{equation}

Now by inequalities (\ref{E41}) and (\ref{E42}) it follows that
$$
\lim_{T\rightarrow \infty} \frac{|\{d<T: d \in S_K\}|}{|\{d<T\}|} =\prod_{\substack{p\\\ \mathfrak{p} \in \mathbb{P}}}(1-A_p/p^2).
$$
\end{proof}

\vspace{.5cm}

\section{Heuristics in support of Conjecture \ref{Con2}}
This section comprises of two subsections; in the first subsection we present the heuristics in support of Conjecture \ref{Con2} and in the second subsection we demonstrate that this method provides a complete solution of Conjecture \ref{Con1} for quadratic polynomials.

\subsection{Heuristics}
It is immediate to notice that the Conjecture \ref{Con2} holds true if the inequality $R_{\ell}(p^2)>p$ holds true for all $p\in S^1(\ell,T)$ with at most $o(T)$ exceptions. Under this assumption
$$|S^2(\ell,T)|<<|\{p \in S^1(\ell,T): R_{\ell}(p^2)>p\}|+o(T)=o(T).$$

We have calculated $R_{\ell}(p^2)$ for small values of $\ell $ and $p$. In particular, for $\ell=5$ we have considered primes $p\in S^1(\ell, T)$ up to  $200000$ and found that the inequality $R_{\ell}(p^2)>p$ holds true with only two exceptions (namely $p=11 \mbox{ and } 131$). In fact, mostly $R_{\ell}(p^2)$ is much larger compared to $p$ and $R_{\ell}(p)$. Similarly, for $\ell=7$ we have considered primes $p<100000$ and the computations show that the inequality $R_{\ell}(p^2)>p$ is true in this range without any exception. For $\ell=11$, we considered primes $p<10000$ and found no exception.\\

Furthermore, for any prime $p$ we can shift the number line such that $R_{\ell}(p^2)>p$ holds true; for example, for any 
$$d \in \bigcup_{k=1}^{\infty} \left(R_{\ell}(p)+(k-1)(p+1)p, R_{\ell}(p)+k(p-1)p\right),$$
if we replace $\N$ by $[d, \infty]$ then $R_{\ell}(p^2)>p$ holds true. What is required is to do this simultaneously for almost all primes. Roughly, it amounts to the set
$$\bigcap_p \bigcup_{k=1}^{\infty} \left(R_{\ell}(p)(k-1)(p+1)p, R_{\ell}(p)+k(p-1)p\right)$$
being non-empty.\\

The assertion $R_{\ell}(p^2)>p$ is equivalent to the assertion that the Diophantine equation
$$x^{\ell-1}+x^{\ell-2}+ \ldots +x+1=z p^y$$
has no solution with $x<p, 1<y \mbox{ and }z\in \Z$.\\

We present one more supporting argument for Conjecture \ref{Con2}. To state our theorem in this regard, we call a prime $p$ to be essential for $\Phi_{\ell}(X)$ if for all integers $c$ satisfying $|cp|<\Phi_{\ell}(p^{2/\ell})$ the polynomials $\Phi_{\ell}(X)-cp$ are irreducible. We prove the following theorem.
\begin{thm}\label{a1ST1}
Assume that the number of essential primes $p \in S^1(\ell, T)$ is at most $o(T)$. Then we have
\begin{equation}\label{NE1}
|S^2(\ell,T)| =o(T).
\end{equation}
\end{thm}

First we prove the following proposition. In order to prove Theorem \ref{a1ST1} we need to consider primes $p \in S^1(\ell, T)$, we remark that we can assume that $p$ is large enough (larger than some fixed number $p_0$).
\begin{prop}\label{a1SP1}
Let $p \equiv 1 \pmod {\ell}$ be a prime and assume that there is an integer $c$ with $|cp|< \Phi_{\ell}(p^{2/{\ell}})$ such that the polynomial $\Phi_{\ell}(X)-cp$ is reducible over $\Q$. For any integer $t$ if there is a prime $\mathfrak{p}$ lying above $p$ such that $\mathfrak{p}^2 | t-\zeta_{\ell}$ then there exist an $\epsilon>0$ for which the following holds $$|R_{\ell}(p)| << t^{1-\epsilon}.$$
\end{prop}
\begin{proof}
Observe that if $\mathfrak{p}^2 | t-\zeta_{\ell}$ then $p^2| \Phi_{\ell}(t)$. Thus, there is some $\epsilon >0$ with $p^{(2/{\ell})}<t^{1-\epsilon}$. Next we shall show that there is some integer $d \in (2,p^{2/{\ell}})$ such that 
$$\Phi_{\ell}(d) \equiv 0 \pmod p,$$
from which it will follow that $$R_{\ell}(p) \leq d <p^{2/ \ell}<t^{1-\epsilon}.$$
Since there are only finitely many primes $p$ with $\Phi_{\ell}(2)>p$, for our purpose we can assume that $\Phi_{\ell}(2)<p$. For any integer $s \geq 1$ with $sp<\Phi_{\ell}(p^{2/{\ell}})$ consider
 $$f(X)=\Phi_{\ell}(X)-sp.$$ 
Clearly $f(2)<0$ and $f(p^{2/{\ell}})>0$. Thus by mean value theorem there is a $\delta \in   (2,p^{2/{\ell}})$ such that $f(\delta)=0$. Since $f(X) \in \Z[X]$ is monic it follows that $\delta $ is an algebraic integer. If $\delta$ is an integer then by taking $d=\delta$ we are through. So we consider the case when $\delta$ is not an integer and put $K=\Q(\delta)$. If we let $D_f$ and $D_{\ell}$ denote the discriminant of $f(X)$ and $\Phi_{\ell}(X)$ respectively then it follows that 
$$D_f \equiv D_{\ell} \pmod p,$$
as roots of $f(X)$ and $\Phi_{\ell}(X)$ are same modulo primes in $L=\Q(\zeta_{\ell}, \delta)$ above $p$.

Let $d_1, \ldots, d_{\ell-1}$ be integers such that 
$$\Phi_{\ell}(X) =\prod_i (X-d_i) \pmod p.$$ Then we also have
$$f(X) =\prod_i (X-d_i) \pmod p.$$
If $\wp$ be a prime in $K=\Q(\delta)$ lying above $p$, then, by Dedekind theorem, we have
$$\wp=<p,d_i-\delta>,$$

$\wp$ has ${\ell}-1$ conjugates. This shows that $f(X)$ is irreducible. \\
But we have assumed that there is some integer $c \geq 1$ with $cp<\Phi_{\ell}(p^{2/{\ell}})$ such that $f(X)$ is reducible. This proves the proposition.
\end{proof}

\begin{proof}[Proof of Theorem \ref{a1ST1}]
We only concentrate on primes $p \in S^1(\ell, T)$ which are not essential for $\Phi_{\ell}(X)$. Let $p_T$ be a prime in $S^1(\ell, T)$ such that  $t=R_{\ell}(p_T)$ is the largest among all $R_{\ell}(p)$ with $p \in S^1(\ell, T)$ running over non-essential prime for $\Phi_{\ell}(X)$. 
Then, $t=R_{\ell}(p_T^2)<T$. Since $p_T$ is not an essential prime and $p_T^2|\Phi_{\ell}(t)$, by Proposition \ref{a1SP1}, there exists an $\epsilon>0$ such that
$$R_{\ell}(p_T)<< t^{1-\epsilon}<T^{1-\epsilon}.$$
In particular for all non-essential primes $p \in S^1(\ell, T)$ we obtain
$$R_{\ell}(p)<< T^{1-\epsilon}.$$
From this we conclude that the number of non-essential primes in $S^1(\ell,T)$ is at most $|S^1(\ell, T^{1-\epsilon})|$

Consequently, we see that
$$|S^2(\ell,T)| << |S^1(\ell, T^{1-\epsilon})|+o(T)=o(T).$$
\end{proof}

\subsection{Quadratic polynomials}
With the work of Estermann \cite{TE}, Ricci \cite{GR}, we know that the Conjecture \ref{Con1} holds for quadratic polynomials. In fact, some better error terms are also known \cite{FI, HH2}. In this subsection we consider a quadratic polynomial and demonstrate that our method does succeed in proving Conjecture \ref{Con1} for this family of polynomials. For simplicity of exposition, we consider the quadratic polynomial  $f(X)=X^2+d$. We remark that we just need to prove an analogue of Conjecture \ref{Con2}, the rest of the proof is same as in Section 4.\\

As in the case of cyclotomic polynomials, we define $$R_f(n)=\min \{d \in \N: n|f(d)\},$$
whenever there exists a $d \in \N$ such that $n|f(d)$ else we put $R_f(n)=\infty$. \\

We claim that $R_f(p^2)>p$ holds true for all but finitely many primes. From this it follows that $|S^2(f,T)|=o(T)$, where
$$S^2(f,T)=\{p: R_f(p^2)<T\}.$$
This immediately establishes an analogue of Conjecture \ref{Con2} for the polynomial $f(X)$.\\

On contrary to our claim, we assume that $R_f(p^2) \leq p$, then we see that $R_f(p^2)=p-k$ for some $k \in \{1, \ldots, p-1\}$. But for the $k$ in given range, $f(p-k)=(p-k)^2+d$ satisfies $f(p-k)<p^2$ for all but finitely many primes $p$. Hence $p^2$ can not divide $f(p-k)$. Consequently, $R_f(p^2)=p-k$ is not possible. This establishes our claim.

\newpage
\section{Appendix}
In this appendix, under the abc conjecture, we show, in an elementary way, that the cyclotomic polynomials take positive proportion of square free values . As mentioned in the introduction, Granville has proved the Conjecture 1 under abc conjecture (see \cite{AG}). Also, before Granville, Browkin et all had demonstrated that under abc conjecture the cyclotomic polynomials take infinitely many square free values (see \cite{BFGS}). We will follow the arguments of \cite{MMS} to show, under abc conjecture, that the cyclotomic polynomials take positive proportion of square free values.\\

The abc conjecture (due to Oesterle, Masser, Szpiro) asserts that, for any fixed $\epsilon>0$ and positive coprime integers $a,b,c$ satisfying $a+b=c$ one has
$$c<<_{\epsilon} N(a,b,c)^{1+\epsilon},$$
where $N(a,b,c)$ is the product of distinct primes dividing $abc$.\\

Let $n$ be an odd integer,  then we will show that $a^n-1$ takes square free values for positive proportion of arguments $a$. For an integer $a$ let $T(a)=1- a^n$. Let $A$ be a large positive real number and let $D(A)$ denote the number of square free integers $d$ with at least one solution to 
\begin{equation}\label{a2SE1}
d=1-a^n, A \leq |a| \leq 2A.
\end{equation}
For a square free integer $d$ let $R(d)$ denote the number of solutions to (\ref{a2SE1}).
\begin{lem}\label{a2l1}
We have $$\sum_d R(d) >> A.$$
\end{lem}
\begin{proof}
Consider $$H(a)=T(a)T(-a)=1-a^{2n},$$ and let $M_1(A)$ denote the set of integers $a$ with $A \leq |a| \leq 2A $ such that $H(a)$ is not divisible by square of any prime $p \leq \log A$. Put $M_1=|M_1(A)|$ and $P=\prod_{p \leq \log A} p$. Let $\mu$ denote the Mobius function then one has
$$\sum_{m^2|(\alpha, P^2)} \mu(m)=1 \mbox{ or }0$$ depending upon whether $p^2 \nmid \alpha$ for all $p \leq \log A$ or otherwise. Thus we get 
\begin{equation}\label{a2SE2}
M_1=\sum_{A \leq |a| \leq 2A} \left( \sum_{m^2|(H(a), P^2)} \mu(m) \right) \\
=\sum_{m|P} \mu(m) \sum_{\substack{A \leq |a| \leq 2A, \\\ m^2|H(a)}} 1.
\end{equation}
Let $$ \varrho(\ell)=\{ b \pmod {\ell} :H(b)=0 \pmod {\ell} \},$$ then $\varrho(\ell)$ is a multiplicative function and for primes $q$ not dividing $2n$ and integers $\alpha$ one has 
$$ \varrho(q^{\alpha})=\varrho(q) \leq 2n.$$
By dividing the last sum in (\ref{a2SE2}) into intervals of length $m^2$ we get
\begin{eqnarray*}
M_1&=& \sum_{m|P} \mu(m) \left( \frac{2A}{m^2} \varrho(m^2)+O(\varrho(m^2)) \right) \\
&=& 2A \sum_{m|P} \mu(m) \frac{\varrho(m^2)}{m^2}+O\left(\sum_{m|P} \mu(m)\varrho(m^2)) \right) \\
&=& 2A \prod_{p|P} \left(1-\frac{\varrho(p)}{p^2} \right) +O(A^{\epsilon})
\end{eqnarray*}
Since the product $ \prod_{p|P} \left(1-\frac{\varrho(p)}{p^2} \right) $ converges as $A \longleftarrow \infty$ we have
$$M_1=2cA+o(A),$$ for some constant $c>0$.\\
Now let $M_2(A)$ denote the set of integers $a$ with $A \leq |a| \leq 2A $ such that $H(a)$ is divisible by square of a prime $p \in (\log A, A]$ and let $M_2=|M_2(A)|$. We have
\begin{eqnarray*}
M_2&=& \sum_{\log A < p \leq A} \sum_{\substack{A \leq |a| \leq 2A \\\ p^2|H(a)}} 1\\
&=&  2A \sum_{\log A < p \leq A}  \frac{\varrho(p^2)}{p^2}+O\left(\sum_{\log A <p \leq A} \varrho(p^2)) \right) \\
\end{eqnarray*}
Since $\varrho(p^2) \leq 2n$, we see that
$$\sum_{\log A< p \leq A} 2A \frac{\varrho(p^2)}{p^2} << 2A \sum_{p> \log A} \frac{1}{p^2}=o(A)$$
and
$$\sum_{\log A < p \leq A} \varrho(p^2) << \frac{A}{\log A} =o(A),$$
by using prime number theorem. Thus $M_2 << o(A).$\\
Note that $M_1(A)\setminus M_2(A)$ consists of integers $a$ with $A \leq |a| \leq 2A $ such that $H(a)$ is not divisible by square of any primes $p \leq A$ and 
$$|M_1(A) \setminus  M_2(A) | \geq M_1-M_2 >> A.$$

We say that an integer $a$ is good if $T(a)$ is not divisible by $p^2$ for any prime $p>A$. We claim that either $a$ or $-a$ is good. If this is not the case then
$$T(a)=p^2r, T(-a)=q^2s,$$ for some integers $r,s$ and some primes $p,q>A$. Certainly $p \neq q$, else we get $p^2|2$. With this
$$T(a)T(-a)=(1-a^n)(1+a^n)=1-a^{2n}=p^2q^2 rs.$$
Now from the $abc$ conjecture
$$p^2q^2rs << (N(a^{2n}, p^2q^2rs,1))^{1+\epsilon}<<\left( \frac{A}{\log A} pq rs \right)^{1+\epsilon}.$$
From this we get $pq<< A^2,$ which is a contradiction to the assumption $p,q>A$. Thus
$$\sum_d R(d) \geq (M_1-M_2)/2 >>A.$$
\end{proof}
Now, assuming the $abc$ conjecture, we prove that cyclotomic polynomials take positive proportion of square-free values.\\
If $$\mbox{ number of square free integers }d \mbox{ which are represented by }T(a) \mbox{ is } >> A, $$ then $1-a^n$ takes positive proportion of square free values and hence the cyclotomic polynomial $\Phi_n(X)$ takes positive proportion of square free values. On the other hand, if 
$$\mbox{ number of square free integers }d \mbox{ which are represented by }T(a) \mbox{ is } << A, $$
then $R(d)\leq n$ and, as in \cite{MMS}, we have the following lemma.
\begin{lem}\label{a2l2}
We have 
$$\sum_d R(d)^2 <<A.$$
\end{lem}

By combining Lemma \ref{a2l1} and Lemma \ref{a2l2} and using Cauchy-Schwarz inequality we get 
$$\mbox{ number of square free integers }d \mbox{ which are represented by }T(a) \mbox{ is } >> A. $$
As argued earlier, it follows that cyclotomic polynomial takes positive proportion of square-free value.

\newpage
\section{Appendix}
Here we illustrate our solution for general quadratic polynomial $f(X)=aX^2+bX+c$. We work with the polynomials $af(X)=(aX)^2+b(aX)+ac$. With a change of variable this is same as $g(Y)=Y^2+bY+c'$. We aim to show that for almost all primes $p$, the relation 
$$R_g(p^2)>p$$
holds true. For a prime $p$ if $R_g(p^2)<p$ then there exist $k>0$ such that $p^2|g(p-k)$. That is,
$$(p-k)^2+b(p-k)+c'=dp^2, \mbox{     for some integer }d.$$
This gives 
$$p^2-(2k-b)p+k^2-bk+c'=dp^2.$$
{\bf Claim: }For $p$ large $|p^2-2kp+bp+k^2-bk+c'|\leq p^2$. \\
For example take $p$ such that $p>8|b|, 8|c'|$. Consider the case when $b>0$. If $2k \leq b$ then $0 \leq -2kp+bp \leq p^2/8$ and $0 \leq k^2 \leq \frac{p^2}{64}$. As $-bk<0$ and $|c'|<p/8$, so
 $|p^2-2kp+bp+k^2-bk+c'|<2 p^2$. Now the claim follows as lhs is an integer multiple of $p^2$. Similarly we get this inequality for $2k>b$ as well. Now the only choices for $d$ are $0,-1,1$ we see that none of these occur for large $p$.


\par
\par

\end{document}